\documentclass[12pt]{amsart}
\usepackage{graphicx}
\usepackage{amsmath}
\usepackage{amssymb}
\usepackage{amsfonts}
\usepackage{verbatim}
\usepackage{scalefnt}
\usepackage{multirow}
\usepackage{mathtools}
\usepackage{float}
\restylefloat{figure}
\usepackage[margin=1.3in]{geometry}

\usepackage{pstricks,tikz}

\usepackage{fancyhdr}
\usepackage{hyperref}

\begin{document}

\def\COMMENT#1{}

\newtheorem{theorem}{Theorem}
\newtheorem{lemma}[theorem]{Lemma}
\newtheorem{proposition}[theorem]{Proposition}
\newtheorem{corollary}[theorem]{Corollary}
\newtheorem{conjecture}[theorem]{Conjecture}
\newtheorem{claim}[theorem]{Claim}
\newtheorem{definition}[theorem]{Definition}
\newtheorem{question}[theorem]{Question}


\def\eps{{\varepsilon}}
\newcommand{\cP}{\mathcal{P}}
\newcommand{\cT}{\mathcal{T}}
\newcommand{\cL}{\mathcal{L}}
\newcommand{\ex}{\mathbb{E}}
\newcommand{\eul}{e}
\newcommand{\pr}{\mathbb{P}}

\title[Local conditions for exponentially many subdivisions]{Local conditions for exponentially many subdivisions}
\author{Hong Liu, Maryam Sharifzadeh and Katherine Staden}
\thanks{H.L.\ was supported by EPSRC grant~EP/K012045/1, ERC
grant~306493 and the Leverhulme Trust
Early Career Fellowship~ECF-2016-523. M.Sh and K.S.\ were supported by ERC grant~306493.}

\begin{abstract}
Given a graph $F$, let $s_t(F)$ be the number of subdivisions of $F$, each with a different vertex set, which one can guarantee in a graph $G$ in which every edge lies in at least $t$ copies of $F$.
In 1990, Tuza asked for which graphs $F$ and large $t$, one has that $s_t(F)$ is exponential in a power of $t$.
We show that, somewhat surprisingly, the only such $F$ are complete graphs, and for every $F$ which is not complete, $s_t(F)$ is polynomial in $t$. Further, for a natural strengthening of the local condition above, we also characterise those $F$ for which $s_t(F)$ is exponential in a power of $t$.

\medskip

\noindent \textbf{Mathematics Subject Classication.} 05C10, 05C35, 05C83.
\end{abstract}


\date{\today}
\maketitle

\section{Introduction and results}

A \emph{subdivision} or \emph{topological minor} of a graph $F$, denoted by $TF$, is a graph obtained by replacing the edges of $F$ with internally vertex-disjoint paths between endpoints.
A classical result of Mader~\cite{mader} from 1972 states that there is some function $f(d)$ of $d$ such that any graph with average degree $f(d)$ contains a copy of $TK_d$. Mader, and independently Erd\H{o}s and Hajnal~\cite{eh}, conjectured that the correct order of magnitude for $f(d)$ is $d^2$.
This was verified by Bollob\'as and Thomason~\cite{bt} and independently by Koml\'os and Szemer\'edi~\cite{ks}.
The current best bound is due to K\"uhn and Osthus~\cite{KO-sub}.

\begin{theorem}[\cite{KO-sub}]\label{btks}
Let $G$ be a graph with average degree at least $(1+o(1))\frac{10}{23} d^2$. Then $G \supseteq TK_d$.
\end{theorem}

This result is tight up to the constant factor; as Jung~\cite{jung} observed, the complete bipartite graph $K_{r,r}$ with $r = \lfloor d^2/8 \rfloor$ does not contain a copy of $TK_{d+1}$. 

We say that two subgraphs $H,H'$ of a graph $G$ are \emph{distinguishable} if $V(H) \neq V(H')$. 
In 1981, Koml\'os~(see~\cite{tuza2}) conjectured that every graph with minimum degree $d$ contains at least as many distinguishable cycles as $K_{d+1}$,\footnote{Koml\'os's conjecture has recently been resolved by J.~Kim and the authors, see~\cite{KLShS}.} that is, $2^{d+1}-{d+1\choose 2}-d-2$. A weaker conjecture, asking whether minimum degree $d$ forces exponential in $d$ many distinguishable cycles, was verified by Tuza~\cite{tuza2} who obtained a lower bound $2^{\lfloor(d+4)/2\rfloor}-O(d^2)$ for graphs with average degree $d$. If instead one imposes a different local condition on $G$: that every edge lies in at least $t$ triangles; it is shown in~\cite{tuza2} that $G$ contains as least as many distinguishable cycles as $K_{t+2}$. 
In other words, $G$ contains at least as many distinguishable subdivisions of the triangle as $K_{t+2}$.

Our aim in this paper, answering a question of Tuza~\cite{tuza2,tuza,tuza3}, is to generalise this from triangles to arbitrary fixed graphs $F$ and establish for which $F$ \emph{locally} many copies of $F$ guarantee exponentially many distinguishable subdivisions of $F$ \emph{globally}.
To formulate this more precisely, let us say that a graph $G$ is \emph{$(F,t)$-locally dense} if every edge $e$ of $G$ lies in at least $t$ copies of $F$.
Write $s(F,G)$ for the size of the largest collection of distinguishable subdivisions of $F$ in $G$.
Define 
$$s_t(F)=\min\{ s(F,G) : G  \mbox{ is } (F,t)\mbox{-locally dense}\}.$$

Tuza's question~\cite{tuza2} was for which graphs $F$ one can find $c,c' > 0$ such that whenever $t$ is a sufficiently large integer compared to $|F|$, we have $s_t(F) \geq c^{t^{c'}}$.
Our main result answers this question completely, showing that the only such $F$ are complete graphs. For every $F$ that is not complete, we construct graphs that are locally dense in $F$, with only polynomially many distinguishable subdivisions of $F$.

\begin{theorem}\label{main}
Let $\ell \geq 3$ be an integer. Whenever $t$ is sufficiently large, the following hold.
\begin{itemize}
\item[(i)] $s_t(K_\ell) \geq 2^{t^{c'}}$, where $c':=c'(\ell) \geq \frac{1}{2(\ell-2)}$.
\item[(ii)] For all graphs $F$ on $\ell$ vertices which are not complete, 
$$s_t(F) \leq  2^\ell(t+2)^{e(F)+2\ell}.$$
\item[(iii)] Whenever $\ell \geq 4$, we have 
$$s_t(K_\ell^-) \geq t^{(1-o(1))(e(K_\ell^-)-2\ell+5)},$$ 
where $K_\ell^-$ is the graph obtained from $K_\ell$ by removing an edge.\footnote{Here $o(1)$ denotes a function which tends to $0$ as $t \to \infty$.}
\end{itemize}
\end{theorem}

For part (i), our bound on $c'(\ell)$ is best possible up to the constant factor $\frac{1}{2}$.
As noted by Tuza~\cite{tuza}, $c'(\ell) \leq \frac{1}{\ell-2}$. This is attained by the graph $G = K_r$, where $r$ satisfies $\binom{r-2}{\ell-2} = t$.
Indeed, every edge of $G$ lies in exactly $t$ copies of $K_\ell$, and every set $X \subseteq V(G)$ of size at least $\ell$ is such that there is a copy of $TK_\ell$ with vertex set $X$.
Therefore,
$$
s_t(K_{\ell},G) = \sum_{i=\ell}^{r} \binom{r}{i} \leq 2^r \leq 2^{t^{\frac{1}{\ell-2}}}.
$$
Nonetheless, we do not believe our bound on $c'(\ell)$ is optimal and so we make no serious attempt to optimise absolute constants.

Part (ii) shows that for all non-complete $F$, the minimum number of distinguishable subdivisions of $F$ among all $(F,t)$-locally dense graphs is at most a polynomial in $t$.
Our upper bound on $s_t(F)$ in (ii) is close to being optimal, as shown by part (iii).


The rest of the paper is organised as follows. In Section~\ref{sec-main}, we give the proof of Theorem~\ref{main}. In Section~\ref{sec-ext}, we study how a natural strengthening of the locally dense condition would change the outcome. Some concluding remarks are given in Section~\ref{sec-conclude}.

\section{Proof of Theorem~\ref{main}}\label{sec-main}

Let us briefly sketch the ideas in the proof.
For part (i), first note that a $(K_\ell,t)$-locally dense graph $G$ has minimum degree which is polynomial in $t$.
Then Theorem~\ref{btks} furnishes us with a subdivision $T$ of $K_d$ in $G$ such that $d$ is polynomial in $t$.
Every subset of the branch vertices of $T$ which has order at least $\ell$ gives rise to a subdivision of $K_\ell$, and all of these are distinguishable.
For part (ii), for each non-complete $F$ we construct a graph $G$ which is $(F,t)$-locally dense but has a very small dominating set $A$ (so each copy of $F$ has very large overlap).
The size of $A$ limits the number of vertices in any subdivision which lie outside of $A$.
The proof of (iii) combines ideas from the previous parts together with an averaging argument.

\medskip
We call those vertices of $TF$ which correspond to those of $F$ the \emph{branch vertices}.\footnote{We use standard graph theoretic notation. For natural numbers $k \leq \ell$, write $[\ell] := \lbrace 1,\ldots, \ell \rbrace$ and write $\binom{[\ell]}{k}$ for the set of $k$-subsets of $[\ell]$.}

\begin{proposition}\label{easy}
Let $k,\ell \in \mathbb{N}$ with $\ell \leq k$ and suppose that $T$ is a subdivision of $K_k$.
Then $T$ contains a subdivision of $K_\ell$ which contains all branch vertices of $T$.
\end{proposition}

\medskip
\noindent
\emph{Proof.}
Let $x_1,\ldots,x_k$ be the branch vertices of $K_k$ in $T$.
Now $T$ is the union of a set $\lbrace P(i,j): ij \in \binom{[k]}{2} \rbrace$ of internally vertex-disjoint paths where $P(i,j)$ has endpoints $x_i, x_j$.
Then we can find a copy of $TK_\ell$ with branch vertices $x_1,\ldots,x_\ell$ by taking the paths $P(i,j)$ for all $\lbrace i,j \rbrace \in \binom{[\ell]}{2} \setminus \lbrace  1,\ell \rbrace$, and taking the concatenation of the paths $P(\ell,\ell+1), P(\ell+1,\ell+2), \ldots, P(k-1,k),P(1,k)$
to obtain a path between $x_1$ and $x_\ell$ which contains every other branch vertex in $V(T)$.
\hfill$\blacksquare$

\medskip
\noindent
Fix $\ell \in \mathbb{N}$.
We will first prove (i).
We may assume that $\ell \geq 4$ since the statement is known for $\ell = 3$ (see \cite{tuza2}).
Let $G$ be $(K_\ell,t)$-locally dense, that is, every edge $xy \in E(G)$ lies in at least $t$ copies if $K_\ell$.
Then
\begin{equation}\label{degcount}
t \le \binom{|N(x) \cap N(y)|}{\ell-2}\le \binom{d(x)}{\ell-2} < \left( \frac{d(x) \cdot e}{\ell-2} \right)^{\ell-2}
\end{equation}
and so $d(x) > e^{-1}(\ell-2)t^{\frac{1}{\ell-2}}$ for all $x \in V(G)$ which are not isolated. Now Theorem~\ref{btks} implies that $G$ contains a copy of $TK_d$, where
\begin{equation}\label{d}
d = (1-o(1))\sqrt{\frac{23(\ell-2)}{10e}}\cdot t^{\frac{1}{2(\ell-2)}}.
\end{equation}
Let $X := \lbrace x_1,\ldots,x_d \rbrace$ be the set of branch vertices of this copy of $TK_d$ and let $P(i,j)$, $ij \in \binom{[d]}{2}$, be the internally vertex-disjoint path between $x_i$ and $x_j$.

Do the following for each $I \subseteq [d]$ with $|I| \geq \ell$.
Let $X_I := \lbrace x_i : i \in I \rbrace$.
Let $T_I$ be the subdivision of $K_{|I|}$ whose set of branch vertices is precisely $X_I$ and which consists of paths $P(i,j)$ with $ij \in \binom{I}{2}$.
So $V(T_I) \cap X = X_I$.
Proposition~\ref{easy} implies that $T_I$ contains a subgraph $S_I$ which is a subdivision of $K_\ell$ such that $V(S_I) \cap X = X_I$.
Thus obtain a collection $\lbrace S_I : I \subseteq [d], |I| \geq \ell \rbrace$ of subdivisions of $K_\ell$ in $G$.
This collection is distinguishable because $V(S_I) \cap X \neq V(S_{I'}) \cap X$ for distinct $I,I' \subseteq [d]$.
So, when $t$ is sufficiently large,
$$
s(F,G) \geq 2^d - \sum_{i=0}^{\ell-1} \binom{d}{i} \geq 2^d - (d+1)^{\ell-1} \geq 2^{t^{\frac{1}{2(\ell-2)}}},
$$
where we used the fact that $\ell \geq 4$ in~(\ref{d}).
This
completes the proof of part (i) of Theorem~\ref{main}.

\medskip
\noindent
We now turn to the proof of (ii), which will follow from the next lemma.

\begin{lemma}\label{semibip}
Let $G$ be a graph with vertex partition $A \cup B$, where $B$ is an independent set.
Let $F$ be a graph on $\ell \geq |A|$ vertices.
Then $s(F,G) \leq 2^{\ell} (|B|+1)^{e(F)+2\ell}$.
\end{lemma}

\medskip
\noindent
\emph{Proof.}
Write $V(F) = [\ell]$.
Let $T$ be a subdivision of $F$ in $G$, and let $x_1,\ldots,x_\ell$ be its branch vertices.
So $T$ consists of pairwise internally vertex-disjoint paths $P(i,j)$ for all $ij \in E(F)$, where $P(i,j)$ has endpoints $x_i,x_j$.
Denote by $\mathcal{P}$ the union of the $P(i,j)$.
Let $\mathcal{P}(A,A)$ be the set of $P(i,j)$ such that $x_i,x_j \in A$, and define $\mathcal{P}(A,B) = \mathcal{P}(B,A)$ and $\mathcal{P}(B,B)$ analogously.
Given $P = P(i,j)$, let $a_P$ be the number of internal vertices of $P$ which lie in $A$, i.e.~$a_P := |A \cap (V(P)\setminus \lbrace x_i,x_j \rbrace)|$ and define $b_P$ analogously for $B$.
Since $B$ is an independent set, the following relationships are easy to see:
\begin{align*}
b_P &\leq a_P+1 \text{ for all } P \in \mathcal{P}(A,A);\\
b_P &\leq a_P \text{ for all } P \in \mathcal{P}(A,B);\\
b_P &\leq a_P-1 \text{ for all } P \in \mathcal{P}(B,B).
\end{align*}
Then, since $|\cP|= e(F)$,
\begin{align*}
\ell \geq |A| \geq \sum_{P \in \mathcal{P}}a_P \geq \sum_{P \in \mathcal{P}}(b_P-1)=\sum_{P \in \mathcal{P}}b_P-|\cP|\quad \Rightarrow \quad \sum_{P \in \mathcal{P}}b_P \le \ell+|\cP|= \ell+e(F).
\end{align*}
Denote by $b^*$ the number of branch vertices of $T$ which lie in $B$. Clearly, $b^*\le \ell$, hence, 
$$|V(T) \cap B| = b^* + \sum_{P \in \mathcal{P}}b_P \leq 2\ell + e(F).$$

Therefore an upper bound for $s(F,G)$ can be obtained by counting the number of subsets of $V(G)$ containing at most $2\ell+e(F)$ vertices in $B$.
Thus
$$
s(F,G) \leq 2^{|A|} \cdot \sum_{i = 0}^{2\ell + e(F)} \binom{|B|}{i} \leq 2^{\ell} (|B|+1)^{e(F)+2\ell},
$$
proving the lemma.
\hfill$\blacksquare$

\medskip
\noindent
We will construct a graph $G$ for which $s(F,G) \leq 2^\ell(t+2)^{e(F)+2\ell}$, via Lemma~\ref{semibip}.
Let $A,B$ be disjoint sets of vertices with $|A|=\ell-2$ and $|B| = t+1$, and let $V(G) = A \cup B$.
Add every edge to $G$ with at least one endpoint in $A$.
For any graph $F$ on $\ell$ vertices which is not complete, we claim that $G$ is $(F,t)$-locally dense.
It suffices to show that $G$ is $(K_\ell^- ,t)$-locally dense. Let $e$ be the non-adjacent pair in $K_\ell^-$ and let $xy \in E(G)$.
If $x,y \in A$, then for any of the $\binom{t+1}{2}>t$ pairs $\lbrace w,z \rbrace$ of vertices in $B$, we have that $G[A \cup \lbrace w,z \rbrace]$ is isomorphic to $K_\ell^-$ with $wz$ playing the role of $e$.
Without loss of generality, the only other case is $x \in A$, $y \in B$.
Then similarly we can choose any of the $t$ vertices $w \in B \setminus \lbrace y \rbrace$ so that $wy$ plays the role of the missing edge $e$.
Therefore every edge $xy \in E(G)$ lies in at least $t$ copies of $K_\ell^-$, and hence $F$.
Now Lemma~\ref{semibip} implies that $s_t(F) \leq s(F,G) \leq 2^\ell(t+2)^{e(F)+2\ell}$, as required.

\medskip
\noindent
Finally, we prove (iii).
Let $\eps > 0$ and let $t$ be sufficiently large compared to $\eps$.
Let $G$ be $(K_\ell^-,t)$-locally dense.
We will show that $s(K_\ell^-,G) \geq t^{(1-2\eps)(e(K_\ell^-)-2\ell+5)}$.
For each edge $e \in E(G)$, let $g(e)$ be the number of copies of $K_{\ell-1}$ in $G$ which contain $e$.
Suppose first that $g(e) \geq t^{\eps}$ for all $e \in E(G)$.
Then $G$ is $(K_{\ell-1},t^{\eps})$-locally dense.
So, by (\ref{d}) and Theorem~\ref{btks}, $G$ contains a copy of $K_d$, where $d = (\sqrt{\ell}/3) t^{\eps/(2(\ell-3))}$.
The same argument as (i) shows that
$$
s(K_\ell^-,G) \geq s(K_\ell,G) \geq 2^d - (d+1)^{\ell-1} \geq 2^{\frac{1}{2}t^{\frac{\eps}{2(\ell-3)}}} > t^{\ell^2}
$$
whenever $t$ is sufficiently large compared to $\eps$ and $\ell$ (and recall that $\ell \geq 4$).

So we may assume that there is some $e \in E(G)$ with $g(e) < t^{\eps}$.

\begin{claim}
There is a subgraph $H$ of $G$ with vertex partition $A \cup B$ where $|A|=\ell-1$ and $|B| \geq 2t^{1-\eps}/(\ell-1)$, such that $H[A]$ is complete and there is $z \in A$ such that $H[A\setminus \lbrace z \rbrace, B]$ is a complete bipartite graph.
\end{claim}

\medskip
\noindent
\emph{Proof.}
Let $\mathcal{J}$ be the set of copies of $K_{\ell-1}$ containing $e$ and let $\mathcal{K}$ be the set of copies of $K_{\ell}^-$ containing $e$.
So $|\mathcal{J}| = g(e)$ and $|\mathcal{K}| \geq t$.
Now, every copy of $K_\ell^-$ containing $e$ contains exactly two distinct $J \in \mathcal{J}$ as subgraphs.
So, by averaging, there is some $J \in \mathcal{J}$ such that there is $\mathcal{K}' \subseteq \mathcal{K}$ with the property that $J \subseteq K$ for all $K \in \mathcal{K}'$ and
\begin{equation}\label{K}
|\mathcal{K}'| \geq \frac{2|\mathcal{K}|}{g(e)}\ge 2t^{1-\eps}.
\end{equation}
Every $K \in \mathcal{K}'$ has exactly one non-adjacent pair, and exactly one member of this pair lies in $V(J)$.
Averaging once again we see that there is some $z \in V(J)$ and $\mathcal{K}'' \subseteq \mathcal{K}'$ such that for all $K \in \mathcal{K}''$ every $y \in V(K)$ is incident with every $x \in V(J) \setminus \lbrace z \rbrace$; and
$$
|\mathcal{K}''| \geq \frac{|\mathcal{K}'|}{\ell-1} \stackrel{(\ref{K})}{\geq} \frac{2t^{1-\eps}}{\ell-1}.
$$
Now, let $A := V(J)$ and let $B := \bigcup_{K \in \mathcal{K}''}V(K) \setminus V(J)$.
Every $K \in \mathcal{K''}$ contains precisely one vertex $x_K$ outside of $V(J)$ and the only non-adjacent pair in $K$ is $x_Kz$.
Therefore the vertices $x_K$ are distinct for distinct $K$.
So $|B| = |\mathcal{K}''|$, as required.
The remaining properties are clear.
\hfill$\blacksquare$

\medskip
\noindent
It will suffice to give a lower bound for $s(K_\ell^-,H)$.
We will count the number of subdivisions $T$ of $K_\ell^-$ in $H$ with specific properties.
Label the vertices of $K_\ell^-$ by $1,\ldots,\ell$, where $E(K_\ell^-) = \binom{[\ell]}{2}\setminus \lbrace \ell-1,\ell \rbrace$.
Write $A = \lbrace x_1,\ldots,x_{\ell-1} := z \rbrace$.
Let $x_\ell$ and $x_{ij}$ for all $ij \in \binom{[\ell-2]}{2}$ be arbitrary distinct vertices in $B$ and let $X$ be the set consisting of these vertices.
Using these we can find a subdivision $T(X)$ of $K_\ell^-$ in $H$ with vertex set $A \cup X$, branch vertices $x_1,\ldots,x_\ell$ and paths $P(i,j)$ for all $ij \in E(K_\ell^-)$, as follows.
\begin{itemize}
\item For all $i \in [\ell-1]$, the vertex $x_i \in A$ corresponds to $i \in V(K_\ell^-)$, and the vertex $x_\ell \in B$ corresponds to $\ell \in V(K_\ell^-)$.
\item For all $ij \in \binom{[\ell-2]}{2}$, each $P(i,j)$ has precisely one internal vertex $x_{ij}$, which lies in $B$. For all $i \in [\ell-2]$ and $j \in \lbrace \ell-1,\ell \rbrace$, $P(i,j)$ is the edge $x_ix_j$. 
\end{itemize}
Note that $T(X)$ exists because $H[A]$ is complete and $H[A \setminus \lbrace x_{\ell-1} \rbrace,B]$ is complete bipartite. Different choices of $X$ give rise to distinguishable $T(X)$ since $V(T(X)) = A \cup X$.
Each $X$ has order $\binom{\ell-2}{2}+1$, so
\begin{eqnarray*}
s(K_\ell^-,G) &\geq& s(K_\ell^-,H) \geq \binom{|B|}{\binom{\ell-2}{2}+1} \geq \left(\frac{2t^{1-\eps}}{(\ell-1)\left(\binom{\ell-2}{2}+1\right)}\right)^{\binom{\ell-2}{2}+1}\\
& \geq &t^{(1-2\eps)\left(\binom{\ell-2}{2}+1\right)}= t^{(1-2\eps)(e(K_\ell^-)-2\ell+5)},
\end{eqnarray*}
completing the proof of (iii).
This concludes the proof of Theorem~\ref{main}.
\hfill$\square$

\section{A spectrum of local conditions}\label{sec-ext}
We now investigate a spectrum of progressively stronger conditions for host graphs $G$, and characterise those graphs $F$ such that every graph $G$ satisfying the condition contains exponentially many copies of $TF$.
Given a graph $F$ on $\ell$ vertices, for each $3 \leq k \leq \ell$, we make the following definition. 
We say that a graph $G$ is \emph{$(F,k,t)$-locally dense} if every edge of $G$ lies in at least $t$ copies of $F'$ for all subgraphs $F' \subseteq F$ with $|F'| \geq k$.
Let $s_t(F,k)$ be the minimum of $s(F,G)$ over all $(F,k,t)$-locally dense graphs $G$.

The next theorem generalises Theorem~\ref{main} by describing the set of all graphs $F$ on $\ell$ vertices for which one can find $c,c' > 0$ such that whenever $t$ is a sufficiently large integer compared to $\ell$, we have $s_t(F,k) \geq c^{t^{c'}}$.
Note that $(F,\ell,t)$-locally dense is the same as $(F,t)$-locally dense.
So this notion seems to be the most natural strengthening of $(F,t)$-locally dense.
As expected, the family of those $F$ of order $\ell$ giving rise to exponentially many subdivisions when the host graph is $(F,k,t)$-locally dense gets strictly larger as $k$ decreases.
In fact $F$ lies in this family if and only if $F \supseteq K_k$.

\begin{theorem}\label{general}
Let $\ell,k \in \mathbb{N}$ and $3 \leq k \leq \ell$. Then, whenever $t$ is sufficiently large, the following hold.
\begin{itemize}
\item[(i)] For all graphs $F$ containing a copy of $K_k$, we have that $s_t(F,k) \geq 2^{t^{c'(k)}}$, where $c'(k) \geq \frac{1}{2(k-2)}$.
\item[(ii)] For all $K_k$-free graphs $F$, we have that $s_t(F,k) \leq 2^\ell(t+2)^{e(F)+2\ell}$.
\end{itemize}
\end{theorem}

\begin{proof}
For (i), let $F$ be a graph on $\ell$ vertices containing a copy of $K_k$.
Let $G$ be a $(F,k,t)$-locally dense graph.
Then $G$ is certainly $(K_k,t)$-locally dense.
The conclusion follows immediately from Theorem~\ref{main}(i).

For (ii), note that $F$ is $K_k$-free if and only if every subgraph $F' \subseteq F$ with $|F'| \geq k$ is such that $F'$ contains a pair of non-adjacent vertices.
Therefore the graph $G$ we constructed in the proof of Theorem~\ref{main}(ii) is in fact $(F,k,t)$-dense.
(To see this, one can use the same argument, as there is always a non-adjacent pair in $F'$ which can we embed into $B$ in at least $t$ different ways.) The conclusion follows from Lemma~\ref{semibip}.
\end{proof}

\section{Concluding remarks}\label{sec-conclude}

We have shown that among all graphs $G$ which are $(F,t)$-locally dense, for large $t$ the minimum number of distinguishable subdivisions of $F$ in $G$ is exponential in a power of $t$ if and only if $F$ is a complete graph.
Whenever $F$ contains a non-adjacent pair of vertices, there is a simple construction of an $(F,t)$-locally dense graph $G$ of order $O(t)$ with a very small dominating set $A$ of order less than $|F|$.
This property means that any subdivision of $F$ in $G$ cannot contain many vertices outside of $A$ which in turn implies $s(F,G)$ is polynomial in $t$.

When one strengthens the local condition on $G$ (in order to increase the family of graphs $F$ which give rise to exponentially many subdivisions) in the most natural way, it is easy to characterise this family using our earlier results.
Therefore it would be of interest to restate the question of Tuza with an entirely different local condition.
To be interesting, such a condition on the host graph $G$ cannot imply for all $F$ that $G$ has minimum degree at least some polynomial in $t$; otherwise Theorem~\ref{btks} and the argument in Theorem~\ref{main}(i) imply the desired conclusion.

We remark that our proof of Theorem~\ref{main} also works when `edge' is replaced by `vertex' in our local condition: that is, we consider host graphs $G$ in which every vertex lies in at least $t$ copies of $F$.
%


\medskip

{\footnotesize \obeylines \parindent=0pt

\begin{tabular}{ll}
Hong Liu, Maryam Sharifzadeh and Katherine Staden            \\
Mathematics Institute		  		 \\
University of Warwick\\
Coventry                             			\\
CV4 2AL				  			\\
UK						     \\
\end{tabular}
}

\begin{flushleft}
{\it{E-mail addresses}:
\tt{$\lbrace$h.liu.9, k.l.staden, m.sharifzadeh$\rbrace$@warwick.ac.uk.}}
\end{flushleft}

\end{document}